\documentclass[reqno,12pt]{amsart}

\include{environ}

\usepackage[usenames]{color}
\usepackage{multirow}
\usepackage{bbm}
\newcommand{\mystrut}{\vrule height9.5pt depth1.5pt width0pt}

\newcommand{\tab}{\par\noindent\mystrut}
\newcommand{\tabb}{\tab\hskip1.5em}
\newcommand{\tabbb}{\tab\hskip3.5em}

\newcommand{\END}{\textbf{end}}
\newcommand{\FOR}{\textbf{for}}

\newcommand{\ELSE}{\hskip2pt\textbf{else}\hskip2pt}

\newcommand{\IF}{\textbf{if}\hskip2pt}
\def\Re{\mathbb{R}}

\makeatletter
\newcommand{\defined}{\mathop{\,{\scriptstyle\stackrel{\triangle}{=}}}\,}
\newcommand{\minimize}[1]{{\displaystyle\minim_{#1}}}
\newcommand{\minim}{\mathop{\operator@font{minimize}}}
\newcommand{\bgap}{\;\;\;}
\newcommand{\mgap}{\;\;}
\newcommand{\PCG}{{\small PCG}}
\newcommand{\CG}{{\small CG}}
\newcommand{\LBFGS}{{\small L-BFGS}}
\newcommand{\MATLAB}{{\small MATLAB}}
\newcommand{\BFGS}{{\small BFGS}}
\newcommand{\etal}{et al.}  
  
\newcommand{\CUTEr}{{\small CUTE}r}
\newcommand{\Qscr}{{\mathcal Q}}
\newcommand{\subject}{\mathop{\operator@font{subject\ to}}} 
\makeatother

\input{jen_macros.sty}

\begin{document}
 \newtheorem{theorem}{Theorem}
\newtheorem{corollary}{Corollary}
\newtheorem{lemma}{Lemma}
\markboth{Erway, Jain, and Marcia}{Shifted {L-BFGS} Systems}


\title{Shifted L-BFGS Systems}

\author{Jennifer B. Erway}
\email{erwayjb@wfu.edu}
\address{Department of Mathematics, Wake Forest University, Winston-Salem, NC 27109}

\author{Vibhor Jain}
\email{vjain@ucmerced.edu}
\address{Appied Mathematics, University of California, Merced, Merced, CA 95343}

\author{Roummel F.\ Marcia}
\email{rmarcia@ucmerced.edu}
\address{Appied Mathematics, University of California, Merced, Merced, CA 95343}

\date{\today}


\thanks{This work was supported
in part by NSF grants DMS-08-11106 and DMS-09-65711.}

\maketitle

\begin{abstract}
  We investigate fast direct methods for solving systems of the form
  $(B+G)x=y$, where $B$ is a limited-memory BFGS matrix and $G$ is a
  symmetric positive-definite matrix.  These systems, which we refer to as
  shifted L-BFGS systems, arise in several settings, including
  trust-region methods and preconditioning techniques for interior-point
  methods.  We show that under mild assumptions, the system $(B+G)x=y$
  can be solved in an efficient and stable manner via a recursion that
  requies only vector inner products.  We consider various shift matrices
  $G$ and demonstrate the effectiveness of the recursion methods in
  numerical experiments.

\end{abstract}


\section{Introduction}
This paper proposes  a recursion formula for solving symmetric
positive-definite shifted limited-memory \BFGS{} (\LBFGS) systems of
equations, i.e., equations of the form
\begin{equation}\label{eqn-main}
(B_k+G)x=y,
\end{equation}
where $B_k$ is a \LBFGS{} matrix and $G$ is a symmetric positive-definite
matrix such that (i) the smallest eigenvalue of $G$ is bounded away
from zero, and (ii) solves with $G+\alpha I$, where $\alpha > 0$, are
efficient and stable.

Systems of the form (\ref{eqn-main}) arise in both constrained and
unconstrained optimization.  In trust-region methods for minimizing
a twice-continuously differentiable function $f$,
the $j$th two-norm trust-region subproblem is given by
\begin{equation} \label{eqn-trustProblem}
     \minimize{s\in\Re^n}\mgap\Qscr (s) \equiv g_j^Ts + \frac{1}{2} s^TH_j s
\bgap\subject            \mgap  \|s\|_2 \le \delta_j,
\end{equation}
where $g_j\defined \nabla f(x_j)$, $H_j \defined \nabla^2 f(x_j)$, and $\delta_j$
is the trust-region radius. \LBFGS{} quasi-Newton trust-region methods
approximate $H_j$ with an \LBFGS{} quasi-Newton matrix $B_j$ (e.g., 
~\cite{Pow70c, Pow70d, DenM77, ByrNS94, Kauf99, Ger04}).  
In this context, $s^*$ is a global solution to the trust-region subproblem
if and only if $\|s^*\|_2\leq \delta_j$ and there exists a unique
$\sigma^*\ge 0$ such that
\begin{equation}                              \label{eqnUC-TR-optimality}
  (B_j+\sigma^* I)s^* = - g,  \quad \text{and} \quad
   \sigma^*(\delta_j - \|s^* \|_2)=0.
\end{equation}
Since $B_j$ is symmetric positive-definite, the system matrix in
(\ref{eqnUC-TR-optimality}) is symmetric positive-definite, and thus, the
matrix-vector equation is an example of a shifted \LBFGS{} system of the
form (\ref{eqn-main}).  In small-scale optimization, trust-region methods
use matrix factorizations to find a pair $(s^*,\sigma^*)$ that satisfy
(\ref{eqnUC-TR-optimality}); in particular, the Mor\'{e}-Sorenson direct
method, arguably the best direct solver, makes use of Cholesky
factorizations of the shifted (approximate) Hessian to find a global
solution of the trust-region subproblem~\cite{MorS83}.  Being able to
efficiently solve shifted \LBFGS{} systems enables the use of direct
methods such as the Mor\'{e}-Sorensen direct method for large-scale
unconstrained optimization.

In constrained optimization, shifted \LBFGS{} systems result when
preconditioning primal-dual penalty and interior-point methods; more
generally, these systems can arise in the context of KKT systems or
saddle-point systems.  For example, consider the following system of
equations
\begin{equation}\label{eqn-kkt}
\begin{pmatrix}
H &\ -A^T \\ A & \ D
\end{pmatrix}
\begin{pmatrix}
x_1 \\ x_2
\end{pmatrix}
= \begin{pmatrix}
b_1 \\
b_2
\end{pmatrix},
\end{equation}
where $A$ is an $m\times n$ matrix, H is symmetric, and $D$ is symmetric
positive definite.  Systems of this form are often called ``KKT systems''
or ``saddle-point systems''.  The equivalent \emph{doubly-augmented
  system}~\cite{ForGG07} is 
\begin{equation}\label{eqn-doubly}
\begin{pmatrix}
H+2A^TD^{-1}A & \ A^T \\ A & \ D
\end{pmatrix}
\begin{pmatrix}
x_1 \\ x_2
\end{pmatrix}
= \begin{pmatrix}
b_1 + 2A^TD^{-1}b_2 \\
b_2
\end{pmatrix},
\end{equation}
and arises in the Newton equations for primal-dual augmented Lagrangian
methods (see, e.g.,~\cite{ForG98, ErwGG09, GilR10}).  Systems of the form
(\ref{eqn-doubly}) also arise in the Newton equations associated with
primal-dual interior-point methods (see, e.g.,~\cite{ForG98, ForGW02,
  ForGG07, ErwGG09}).  In both applications, typically $H$ is the Hessian
of the Lagrangian, $A$ is the constraint Jacobian, and $D$ is a
positive-definite diagonal matrix, which serves as a regularization
(\cite{ErwGG09, GilR10}).

If the matrix $H+2A^TD^{-1}A$ is positive definite, the system matrix in
(\ref{eqn-doubly}) is positive definite~\cite{ForGG07}; thus, preconditioned
conjugate-gradients (\PCG) may be used to solve (\ref{eqn-doubly}).
Forsgren \etal~\cite{ForGG07} recommend a block preconditioner of the form
$$
P=
\begin{pmatrix}
B+2A^TD^{-1}A & \ A^T \\ A & \ D
\end{pmatrix},
$$
where $B$ is an approximation of $H$ such that
$B+2A^TD^{-1}A$ is positive definite.  One benefit of this preconditioner
is that efficient solves with $P$ can be computed provided solves
with $B+2A^TD^{-1}A$ are efficient.  To see this, note that
\begin{equation}\label{eqn-interiorP}
\begin{pmatrix}
B+2A^TD^{-1}A & \ A^T \\ A & \ D
\end{pmatrix}
\begin{pmatrix}
v_1 \\ v_2
\end{pmatrix} =
\begin{pmatrix}
r_1 \\ r_2
\end{pmatrix}
\end{equation}
is equivalent to first solving $(B+A^TD^{-1}A)v_1 = r_1-A^TD^{-1}r_2,$ for
$v_1$ and then directly computing $v_2 = D^{-1}(r_2-Av_1)$. In the case
when $A$ is a constant positive-definite diagonal matrix (e.g., the
constraints are simple bounds), solves with $P$ are efficient whenever
solves with $B+G$ are efficient, where $G$ is a symmetric positive-definite
diagonal matrix.  In large-scale optimization, \LBFGS{} matrices are a
common choice for approximating matrices of unknown structure.  If $B$ is
taken to be an \LBFGS{} approximation to $H$, then the resulting system to
be solved is a shifted \LBFGS{} system.

\bigskip

In this paper we investigate fast direct methods for solving shifted
\LBFGS{} systems where the shift $G$ is a symmetric positive-definite
matrix.  Recent work has developed formulas for the case when $G$ is a
diagonal matrix~\cite{ErwayMarcia12, ErwayMarciaWCE12}; however, no
stability proof was given for the proposed recursion formulas.  In this
paper, we derive recursion formulas for the cases when $G$ is sufficiently
positive-definite (i.e., the smallest eigenvalue of $G$ is bounded below)
and solves with $G+\alpha I$, where $\alpha$ is a constant, are efficient
and stable.  An important contribution of this paper is a stability proof for
the proposed recursion formulas that includes the case when $G$ is a
positive diagonal matrix.

This paper is organized in six sections.  Section 2 is a review of
\LBFGS{} updates, including the famous two-term recursion
formula~\cite{Noc80} for solves with the \LBFGS{} matrix.  Section 3
introduces shifted \LBFGS{} systems and reviews the recursion formula
for shifted \L-BFGS{} systems.  In Section 4, we present stability
results.  We demonstrate the effectiveness of the recursion methods in
Section 5.  Future research directions and conclusions are found in
Section 6.


\section{Limited-memory BFGS matrices}
In this section, we review \LBFGS{} matrices and their updates.

\medskip

The \BFGS{} matrix is defined by a sequence of pairs of updates
$\{(s_i,y_i)\}$ as follows:
\begin{equation}\label{eqn-pairs}
	s_i=x_{i+1}-x_i\quad  
	\textrm{and} \quad 
	y_i=\nabla f(x_{i+1})-\nabla f(x_i).
\end{equation}

The initial \BFGS{} matrix is taken to be a scalar multiple of the identity,
i.e., $B_0=\gamma_k^{-1}I$.  For each pair $(s_i,y_i)$, the quasi-Newton
matrix is updated as follows:
\begin{equation}\label{eqn:Bi+1}
B_{i+1} = B_i -   \frac{1}{s_i^TB_is_i}B_i s_is_i^TB_i +
		 \frac{1}{y_i^Ts_i}y_iy_i^T.
\end{equation}
Each \BFGS{} matrix $B_i$ is symmetric positive definite. In practice, the
initial matrix $B_0$ is often taken to be $\gamma_k\equiv
s_{k-1}^Ty_{k-1}/\|y_{k-1}\|_2^2$, where $k$ is the index of the last stored pair
(see, e.g.,~\cite{LiuN89} or~\cite{Noc80}).

\bigskip

In large-scale optimization, it is advantageous to store only a few of the most
recent pairs $\{(s_i,y_i)\}$, i.e., typically less than ten (Byrd
\etal{}~\cite{ByrNS94} recommend between two and six).  In this case, the
\BFGS{} matrix is referred to as a limited-memory \BFGS{} (\LBFGS) matrix.
For this paper, we consider \LBFGS{} matrices where $M$ denotes the maximum
number of stored \LBFGS{} updates.

One advantage of \LBFGS{} updates is that for solving systems of the form
$B_kx=r$ there is a two-loop recursion formula~\cite{Noc80, NocW06}.
\newline \newline
\begin{Pseudocode}{Algorithm 2.1: Two-loop recursion to compute $x=B_k^{-1}r$}\label{alg-recursion}
  \tab $q\leftarrow r$;
  \tab \FOR\ $i=k-1,\dots,0$
  \tabb $\alpha_i\leftarrow(s_i^Tq)/(y_i^Ts_i)$;
  \tabb $q \leftarrow q-\alpha_iy_i$;
  \tab \END\ 
  \tab $x\leftarrow B_0^{-1} q$;
  \tab \FOR\ $i=0,\ldots, k-1$
  \tabb  $\beta \leftarrow (y_i^Tx)/(y_i^Ts_i)$;
  \tabb  $x \leftarrow x+(\alpha_i-\beta)s_i$:
  \tab \END\
\end{Pseudocode}

To solve shifted \LBFGS{} systems, one might be tempted to make a simple
change to Algorithm 2.1, particularly by replacing $x\leftarrow B_0^{-1}q$ with 
$x\leftarrow (B_0+G)^{-1}q$ in between the two loops in Algorithm 2.1.  
However, this does not yield a solution to $(B_k+G)x=r$. It is easiest to see this
by considering the first update.  Let $\widehat{B}_0 = (B_0 + G)$.  Then
the two-loop recursion where the $x\leftarrow B_0^{-1}q$ step is replaced by 
$x\leftarrow \widehat{B}_0^{-1}q$ would yield a solution that satisfies
$x = \widehat{B}_1^{-1}r$, where
\begin{eqnarray*} \nonumber
	\widehat{B}_1 
	&=& \widehat{B}_0 - \frac{1}{s_0^T\widehat{B}_0s_0} \widehat{B}_0 s_0 s_0^T\widehat{B}_0 +
			\frac{1}{y_0^Ts_0} y_0 y_0^T\\ 
	&=&  (B_0 + G) - \frac{1}{s_0^T\widehat{B}_0s_0} \widehat{B}_0 s_0  s_0^T \widehat{B}_0+
			\frac{1}{y_0^Ts_0} y_0 y_0^T \nonumber \\ 
	&=& \left ( B_0 - \frac{1}{s_0^T\widehat{B}_0s_0} \widehat{B}_0 s_0 s_0^T\widehat{B}_0  +
			\frac{1}{y_0^Ts_0} y_0 y_0^T  \right ) + G. \nonumber 
\end{eqnarray*}
Note that $\widehat{B}_1 \ne B_1 + G$.  
Thus, $x$ satisfies $\widehat{B}_1x = r$; however, $x$ does not satisfy
the shifted \LBFGS{} system, i.e.,
$(B_1 + G)x\neq r$.



For the duration of the paper we assume that at most $M$ pairs
$\{(s_i,y_i)\}$, $i=0,\ldots M-1$ are stored at any given time.  Moreover,
we assume the usual requirement that a pair $(s_i,y_i)$, $i=0,\ldots M-1$,
must satisfy $y_i^Ts_i>0$ in order to guarantee that each $B_i$ is positive
definite.  For $k<M$, the $k$th vectors $s_k$ and $y_k$ are stored as the
$k$th column in $S$ and $Y$, respectively.  We use Algorithm 2.2 to update
the matrices $S$ and $Y$ as new pairs $(s_k,y_k)$ are generated; thus, at
all times we have exactly $k$ stored vectors with $k\le M-1$.

\begin{Pseudocode}{Algorithm 2.2: Update $S$ and $Y$}\label{alg-update}
\tab \IF\ $k<M-1$,
\tabb $S\leftarrow [S \,\,\, s_k]; \mgap Y\leftarrow [Y \,\,\, y_k]; \mgap
k\leftarrow k+1$;
\tab \ELSE\ 
\tabb \FOR\ $i=0,\ldots k-1$
\tabbb  $s_i\leftarrow s_{i+1}; \mgap y_i\leftarrow y_{i+1}$; 
\tabb \END\
\tabb $S\leftarrow [s_0,\dots s_{k-1}]; \mgap Y\leftarrow [y_0,\dots y_{k-1}]$; 
\tab \END\
\end{Pseudocode}

\section{Shifted L-BFGS Methods}

Consider the problem of finding the inverse of $B_k+G$, where $B_k$ is an
\LBFGS{} quasi-Newton matrix and $G$ is a symmetric positive-definite
matrix.  We assume that solving systems involving $G+\alpha I$ where $\alpha$
is a scalar can be done in an efficient and stable manner.
In particular, the proposed method is suitable in cases where
$G$ is banded (e.g., diagonal or tridiagonal),
structured (e.g., triadic or circulant), or factorized (e.g., its Cholesky or LDL$^\textrm{T}$ 
decomposition is known).

The Sherman-Morrison-Woodbury (SMW) formula gives the following formula for
computing the inverse of $A+uu^T$, where $A$ is an $n\times n$ symmetric
and invertible matrix and $uu^T$ is a symmetric rank-one update
(see~\cite{GolV96}) with $u\in\mathbb{R}^n$:
\begin{equation*}
(A+uu^T)^{-1} =  A^{-1}-\frac{1}{1+u^TA^{-1}u}A^{-1}uu^T\! A^{-1}.
\end{equation*}

We will now use the SMW formula for computing the inverse of $B_k + G$.
(This discussion closely follows \cite{ErwayMarcia12}.)
First, for $0 \le j < k$, let 
\begin{equation}\label{eqn:u}
	u_{2j} \ = \ 
	\frac{1}{s_j^TB_js_j}B_j s_j \quad
	\text{and}  \quad
	u_{2j+1} \ = \
	 \frac{1}{y_j^Ts_j} y_j
\end{equation}
and let $U_i = (-1)^{i+1}u_i u_i^T$, 
for $0 \le i < 2k$, 
be the rank-one updates in \eqref{eqn:Bi+1}.
Letting $C_0 = (\gamma_k^{-1}I + G)$, we define the matrices
\begin{equation}\label{eqn:Ci}
	C_1 = C_0 + U_0 \ , 
	\quad C_2 = C_1 + U_1 \ ,  \quad C_3 = C_2 + U_2 \ , \dots  \, .
\end{equation}
We compute $(B_k + G)^{-1}z$ by noting that $C_{2k} = B_k + G$ and
applying the SMW formula to $C_{2k}^{-1}z$ recursively: For $0 \le i < 2k$,
\begin{equation}\label{eqn:Ciinv}
	C_{i+1}^{-1} z = C_{i}^{-1}z - 
	\frac{1}{1 + (-1)^{i+1} u_i^TC_i^{-1}u_i} C_{i}^{-1}U_{i}C_{i}^{-1}z,
\end{equation}
and thus, recursively applying \eqref{eqn:Ciinv} to $C_i^{-1}z$, we obtain
\begin{equation}\label{eqn:Ck}
	C_{2k}^{-1}z = C_0^{-1}z
	- \sum_{i=0}^{2k-1}\frac{1}{1 + (-1)^{i+1}u_i^TC_i^{-1}u_i} C_{i}^{-1}U_{i}C_{i}^{-1}z.
\end{equation}

We now show that \eqref{eqn:Ck} can be computed efficiently using only
vector inner products.  Let $$\tau_{i} = \frac{1}{1 +(-1)^{i+1}u_iC_i^{-1}u_i},$$
and let $p_j = C_j^{-1}u_j$.
Then
$
	(C_i^{-1}U_i C_i^{-1})z = (-1)^{i+1}(p_ip_i^T) z = (-1)^{i+1}(p_i^T z) p_i,
$
and \eqref{eqn:Ck} simplifies to 
\begin{equation} \label{eqn-Ck1}
	C_{2k}^{-1}z \ = \ C_0^{-1}z + \sum_{i=0}^{2k-1} (-1)^i \tau_i (p_i^Tz) p_i.
\end{equation} 
We assume that $C_0^{-1}z = (\gamma_k^{-1}I+G)^{-1}z$ is easy to compute.
Therefore, the bulk of the computational effort in forming $C_k^{-1}z$
involves the inner product of $z$ with the vectors $p_i$ for $i = 0, 1,
\dots, k$.  What remains to be shown is how to compute $\tau_i$ and $p_i$
efficiently.

First, 
$\tau_i$ simplifies to
$$
	\tau_i = \frac{1}{1 +  (-1)^{i+1}u_i^Tp_i}.
$$
Next, we can compute $p_i= C_i^{-1}u_i$ by evaluating (\ref{eqn-Ck1})
at $z = u_i$:
$$
	p_i \ = \
	C_0^{-1}u_i +
	\sum_{j = 0}^{i-1}(-1)^j \tau_j(p_j^Tu_i)p_j
$$
Thus, computing and storing $u_i^Tp_j$ enable us to easily compute $\tau_i$ and $p_i$.
Algorithm 3.1 computes
$C_{2k}^{-1}z$ in \eqref{eqn:Ci} using \eqref{eqn:Ck} and \eqref{eqn-Ck1}:
 
\bigskip
\begin{Pseudocode}{Algorithm 3.1: Proposed recursion to compute $x = C_{2k}^{-1}r = (B_k+G)^{-1}r$}\label{alg-recursionC}
$x \leftarrow (\gamma_{k}^{-1}I +
  G)^{-1} r$; \tab \FOR\ $j = 0, \dots, 2k-1$ \tabb \IF\ $j$ even
  \tabbb $u \leftarrow
  \frac{B_{j/2}^{\phantom{T}}s_{j/2}^{\phantom{T}}}
       {\sqrt{s_{j/2}^TB_{j/2}^{\phantom{T}}s_{j/2}^{\phantom{T}}}}$;
       \tabb \ELSE \tabbb $u \leftarrow
       \frac{y_{(j-1)/2}^{\phantom{T}}}{\sqrt{y_{(j-1)/2}^Ts_{(j-1)/2}^{\phantom{T}}}}$;
       \tabb \END\ \tabb $p_j \leftarrow (G+\gamma_{k}^{-1}I)^{-1}u$;
       \tabb \FOR\ $i = 0, \dots, j-1$ \tabbb $p_j \leftarrow p_j +
       (-1)^i\tau_i (p_i^Tu) p_i$; \tabb \END\ \tabb $\tau_j
       \leftarrow 1/(1 + (-1)^{j+1}p_j^Tu)$; \tabb $x_{\phantom{j}}
       \leftarrow x + (-1)^{j+1}\tau_j (p_j^Tx) p_j$; \tab \END\
\end{Pseudocode}
\noindent Note that Algorithm 3.1 requires a total of $2k$ matrix
solves to compute $p_j\leftarrow (G + \gamma_k^{-1}I)^{-1}u$
and requires $2k^2+5k+3$
vector inner products (excluding the definition of $u$ each
iteration).  The $u$ updates in Algorithm 3.1 can be computed
efficiently using 
Procedure 7.6 (Unrolling the BFGS formula) in \cite{NocW06},
which requires a total of $k^2+k$ vector inner products.  
Additionally, Algorithm 3.1 requires $2k^2 - 2k + 1$ vector updates.  
Considering
$k$ is generally between 2 and 6, these total counts are relatively
low.

\section{Stability}

It is well-known that the SMW formula for inverting a rank-one update to a
nonsingular matrix can be numerically unstable (see, e.g.,
\cite{Higham2002,Stewart1974,Yip1986}).
In this section, we address how this potential instability in our proposed
recursion approach is mitigated.  To show that the proposed recursion
approach in computing $C_{i+1}^{-1}z$ in \eqref{eqn:Ci} is stable, we first
show that $1 + (-1)^{i+1}u_i^Tp_i$ is sufficiently bounded away from zero.

The potential source of instability is in the computation of $\tau_i$:
\begin{equation}\label{eqn:taui}
  	 \tau_i = 	\frac{1}{1+ (-1)^{i+1}u_i^Tp_i}.
\end{equation}
When $i$ is odd, there is no instability because the denominator in (\ref{eqn:taui}) 
is bounded away from zero since $C_i$ is positive definite: 
$$
	1 + u_i^TC_i^{-1} u_i > 1.
$$ However, when $i$ is
even, subtraction in the denominator of (\ref{eqn:taui}) could cause
catastrophic cancellation.  To show that the proposed recursion formula is
stable, we prove that the denominator in (\ref{eqn:taui}) is bounded away
from zero.   Let $\theta_{\min}$ be a lower bound on the 
eigenvalues of $G$, i.e., $0 < \theta_{\min} \le \lambda(G)$. 

\begin{lemma}\label{lemma:tracebound}
  Let $Y_j = [ \ y_0 \ \ y_1 \ \cdots \ y_j \ ] \in \mathbb{R}^{n \times
    (j+1)}$.  Suppose that $y_{\ell}^T s_{\ell} \ge \delta > 0$ for $\ell=0,\ldots,k-1$, then for $i = 2j$,
\begin{equation}\label{eqn-4.1}
	1 - u_{i}^TC_i^{-1}u_{i} 
	\ \ge \  
	\frac{\theta_{\min}}{\gamma_k^{-1} + \| Y_{j-1} \|_F^2/\delta + \theta_{\min}}
	\ > \ 0,
\end{equation}
where
$\theta_{\min}$ is a lower bound on the 
eigenvalues of $G$, i.e., $0 < \theta_{\min} \le \lambda(G)$.
\end{lemma}

\begin{proof} Using the definition of $u_i$ given in \eqref{eqn:u}, for $i = 2j$, 
\begin{eqnarray} \label{eqn:lemma1}
	u_{i}^TC_i^{-1}u_{i} &=&
	\frac{1}{s_j^TB_js_j}\bigg ( s_j^T B_j (B_j + G)^{-1} B_j s_j \bigg ). 
\end{eqnarray}
Letting $q = B_j^{1/2}s_j$ in (\ref{eqn:lemma1}), we obtain 
\begin{eqnarray}
	u_{i}^TC_i^{-1}u_{i}  &=& 
	\frac{q^T(I + B_j^{-1/2}GB_j^{-1/2})^{-1}q}{q^Tq} \nonumber \\
	&\le& \lambda_{\max} \bigg ( (I + B_j^{-1/2}GB_j^{-1/2})^{-1} \bigg ) \nonumber \\
	&=& \frac{1}{\lambda_{\min} \bigg ( I + B_j^{-1/2}GB_j^{-1/2} \bigg ) }\nonumber\\
	&\le& \frac{1}{1 + \lambda_{\min}\bigg ( B_j^{-1/2}GB_j^{-1/2} \bigg  )}
\label{eqn:aca}
\end{eqnarray}
since
$\lambda_{\min} ( I + B_j^{-1/2}GB_j^{-1/2}  )\ge\lambda_{\min}(I)+\lambda_{\min}( B_j^{-1/2}GB_j^{-1/2})
$ by~\cite[Theorem 8.1.5]{GolV96}.  Note that
\begin{eqnarray}
	\lambda_{\min} \bigg ( B_j^{-\frac12}GB_j^{-\frac12} \bigg )  
	= \underset{x \ne 0}{\text{ min }} \frac{x^T B_j^{-\frac12}GB_j^{-\frac12}x}{x^Tx} 
	= \underset{y \ne 0}{\text{ min }} \frac{y^TGy}{y^T B_jy} 
	\ge  
	\frac{\theta_{\min}}{\lambda_{\max}(B_j)}. \label{eqn:brilliant}
\end{eqnarray}
Now we find an upper bound for $\lambda_{\max}(B_{j})$.  Suppose $z \ne 0$,  then
\begin{eqnarray*}
	z^TB_{j}z 
	&=& 
	z^T \! \left ( B_{j-1} - \frac{1}{s_{j-1}^T B_{j-1} s_{j-1}}B_{j-1} s_{j-1} s_{j-1}^T B_{j-1}  
		+ \frac{1}{y_{j-1}^Ts_{j-1}}y_{j-1} y_{j-1}^T  \right ) \! z \\
	&=& 
	z^TB_{j-1}z -  \frac{\left(z^TB_{j-1}s_{j-1}\right)^2}{s_{j-1}^TB_{j-1}s_{j-1}}
		+ \frac{(y_{j-1}^Tz)^2}{ y_{j-1}^Ts_{j-1}}  \\
	&\le&
	z^TB_{j-1} z + \frac{\| y_{j-1} \|_2^2 \| z \|_2^2}{ y_{j-1}^Ts_{j-1}},
\end{eqnarray*}
using the Cauchy-Schwartz inequality.
Applying this recursively and setting $B_0 = \gamma_k^{-1} I$, we obtain
\begin{equation}\label{eqn:lambdaBj}
	\lambda_{\max}(B_{j})
	\ = \
	\underset{z \ne 0}{\max} \ \frac{z^TB_j z}{z^Tz}
	\ \le \ 
	\gamma_k^{-1} + \sum_{\ell = 0}^{j-1} \frac{ \| y_{\ell} \|_2^2}{y_{\ell}^T s_{\ell}}.
\end{equation}
Then, (\ref{eqn:aca}) together with (\ref{eqn:brilliant}) and (\ref{eqn:lambdaBj}) yields
$$
	1 - u_{i}^TC_i^{-1}u_{i}  \ \ge \ 
	\frac{\theta_{\min}}{\lambda_{\max}(B_j) + \theta_{\min}}
        \ \ge \
        \frac{\theta_{\min}}{\gamma_k^{-1} + \sum_{\ell = 0}^{j-1} \frac{ \| y_{\ell} \|_2^2}{y_{\ell}^T s_{\ell}} + \theta_{\min}}.
$$
Finally, since $y_l^Ts_l\ge \delta >0$ for $l\in\{0,\ldots,k\}$, we obtain
$$
	1 - u_{i}^TC_i^{-1}u_{i} 
	\ \ge \
	\frac{\theta_{\min}}{\gamma_k^{-1} + \| Y_{j-1} \|_F^2/\delta + \theta_{\min}}
	\ > \ 0,
$$
as desired.
\end{proof}

\bigskip

\noindent The following theorem shows that computing
$C_{i+1}^{-1}r$ is stable.

\begin{theorem} \label{thm:SLBFGSstability}   Let $Y_j = [ \ y_0 \ \ y_1 \ \cdots \ y_j \ ] \in \mathbb{R}^{n \times
    (j+1)}$.  Suppose that $y_{\ell}^T s_{\ell} \ge \delta > 0$ for $\ell=0,\ldots,k-1$, $\|Y_{j-1}\|_F^2\leq \eta$, and $\gamma_k\theta_{\min}>\epsilon$ for some $\epsilon > 0$.  Provided
  solves with $G+\alpha I$ are stable for $\alpha>0$,
then Algorithm 3.1 for computing $C_{i}^{-1}r$ for $i = 1, 2, \cdots, 2k$
is stable.
\end{theorem}

\begin{proof}
  The proof is by induction on $i$.  Consider the base case $i=1$:
By (\ref{eqn:Ciinv})
we have
\begin{equation*}
  C_1^{-1}r
  \ = \ C_0^{-1}r - \frac{1}{1 - u_0^T C_0^{-1} u_0} C_0^{-1}u_0u_0^TC_0^{-1}r .
\end{equation*}
Substituting in for $u_0$ using \eqref{eqn:u} together with $B_0=\gamma_k^{-1} I$ yields
$$
	u_0^T  C_0^{-1}u_0 
	= 
	\frac{ s_0^T B_0(G+\gamma_k^{-1}I)^{-1}B_0 s_0}{s_0^TB_0s_0}
	 = 
 	\frac{s_0^T (G+\gamma_k^{-1}I)^{-1}s_0}{\gamma_k \ \! s_0^Ts_0}
	 \le 
	\frac{1}{1 + \gamma_k \theta_{\min}} < 
	\frac{1}{1 + \epsilon}
<1.
$$
Therefore, $1 - u_0^TC_0^{-1}u_0 >0$, and $C_1^{-1}r$ is computed in a stable manner.

For the induction step, we assume that computing $C_{i}^{-1}r$ is stable.  Then we need to show computing
$C_{i+1}^{-1}r$ is stable.  Since
\begin{equation}\label{eqn:induction}
	C_{i+1}^{-1} r = C_{i}^{-1}r - 
	\frac{1}{1 + (-1)^{i+1}u_iC_i^{-1}u_i} C_{i}^{-1}U_{i}^{\phantom{1}}
	C_{i}^{-1}r,
\end{equation}
we only need to show that the second term  can be computed in a stable manner.
Using Lemma \ref{lemma:tracebound}, 
\begin{equation}\label{eqn:1trace}
	1 + (-1)^{i+1}u_iC_i^{-1}u_i
	\ \ge \
	\min \left \{ 
		1, 
	\displaystyle{\frac{\theta_{\min}}{\gamma_k^{-1} + \eta/\delta + \theta_{\min}}}
	\right \},
\end{equation}
which implies that $1 + (-1)^{i+1}u_iC_i^{-1}u_i$ is bounded
away from zero, and thus, the denominator in (\ref{eqn:induction}) is bounded
away from zero.  By assumption computing $C_i^{-1}r$ is stable; therefore
computing $C_{i+1}^{-1}r$ using (\ref{eqn:1trace}) is stable.
This completes the proof.
\end{proof}

The bound on $\|Y_{j-1}\|^2_F$ can be enforced during run time.  That is,
when $\| Y_{j-1} \|_F>\sqrt{\eta}$, where $\eta$ is user-defined, the current
pairs $\{(s_j,y_j)\}$ can be discarded and the \LBFGS{} method can be restarted.

\section{Numerical experiments}

We demonstrate the effectiveness of the proposed recursion formula by
solving large shifted \LBFGS{} linear systems.  To generate the
first set of large shifted \LBFGS{} systems, we consider
linear systems arising in the context of optimization, and in particular,
in trust-region methods that seek to satisfy the first equation in
(\ref{eqnUC-TR-optimality}).  The second set of tests consider
very large shifted \LBFGS{} systems generated at random.

\subsection{Solving large shifted L-BFGS systems in optimization}
The first series of tests
arise from the \CUTEr{} test collection (see 
\cite{BonCGT95,GouOT03}).  
The test set was constructed using the \CUTEr{} interactive \texttt{select}
tool, which allows the identification of groups
of problems with certain characteristics.  In our case, the \texttt{select}
tool was used to identify the twice-continuously differentiable
unconstrained problems for which the number of variables can be varied.
This process selected 67 problems:
\texttt{arwhead}, 
\texttt{bdqrtic}, 
\texttt{broydn7d}, 
\texttt{brybnd}, 
\texttt{chainwoo}, 
\texttt{cosine}, 
\texttt{cragglvy}, 
\texttt{curly10}, 
\texttt{curly20}, 
\texttt{curly30}, 
\texttt{dixmaana}, 
\texttt{dixmaanb}, 
\texttt{dixmaanc}, 
\texttt{dixmaand}, 
\texttt{dixmaane}, 
\texttt{dixmaanf}, 
\texttt{dixmaang}, 
\texttt{dixmaanh}, 
\texttt{dixmaani}, 
\texttt{dixmaanj}, 
\texttt{dixmaank}, 
\texttt{dixmaanl}, 
\texttt{dixon3dq}, 
\texttt{dqdrtic}, 
\texttt{dqrtic}, 
\texttt{edensch}, 
\texttt{eg2}, 
\texttt{engval1}, 
\texttt{extrosnb}, 
\texttt{fletchcr}, 
\texttt{fletcbv2}, 
\texttt{fminsrf2}, 
\texttt{fminsurf}, 
\texttt{freuroth}, 
\texttt{genhumps}, 
\texttt{genrose}, 
\texttt{liarwhd}, 
\texttt{morebv}, 
\texttt{ncb20}, 
\texttt{ncb20b}, 
\texttt{noncvxu2}, 
\texttt{noncvxun}, 
\texttt{nondia}, 
\texttt{nondquar}, 
\texttt{penalty1}, 
\texttt{penalty2}, 
\texttt{powellsg}, 
\texttt{power}, 
\texttt{quartc}, 
\texttt{sbrybnd}, 
\texttt{schmvett}, 
\texttt{scosine}, 
\texttt{scurly10}, 
\texttt{scurly20}, 
\texttt{scurly30}, 
\texttt{sinquad}, 
\texttt{sparsine}, 
\texttt{sparsqur}, 
\texttt{spmsrtls}, 
\texttt{srosenbr}, 
\texttt{testquad}, 
\texttt{tointgss}, 
\texttt{tquartic}, 
\texttt{tridia}, 
\texttt{vardim}, 
\texttt{vareigvl} and 
\texttt{woods}.  
The dimensions were selected so that $n\ge 1000$, with a default of
$n=1000$ unless otherwise recommended in the \CUTEr{} documentation. 

Using the default initial starting points for these problems, we ran the
\LBFGS{} method to generate five limited-memory pairs, i.e., $(s_i,y_i)$,
$i=0,\dots, 4$.  Having obtained five \LBFGS{} limited-memory updates without
convergence to a minimizer, the
following shifted \LBFGS{} system was solved:

\begin{equation}\label{eqn-trsp}
(B_5+\sigma I ) s=-g_5,
\end{equation}
where $\sigma$ is a positive scalar and $g_5\defined \nabla f(x_5)$.

In practice $\sigma$ can take on any value positive value at optimality
(see (\ref{eqnUC-TR-optimality})).  Because of this, the value for $\sigma$
was randomly chosen between $(0,1)$. The
following problems did not satisfy the requirement
$\gamma_k\theta_{min}>\epsilon$ (see Theorem \ref{thm:SLBFGSstability})
with $\epsilon=1.0\times 10^{-4}$, and thus, were removed from the test
set: \texttt{arwhead}, 
\texttt{bdqrtic}, 
\texttt{curly10}, 
\texttt{curly20}, 
\texttt{curly30}, 
\texttt{dqrtic}, 
\texttt{liarwhd}, 
\texttt{nondia}, 
\texttt{penalty1}, 
\texttt{penalty2}, 
\texttt{power}, 
\texttt{quartc}, 
\texttt{sbrybnd}, 
\texttt{scosine}, 
\texttt{scurly10}, 
\texttt{scurly20}, 
\texttt{scurly30}, 
\texttt{sinquad}, 
\texttt{sparsine}, 
\texttt{testquad}, 
\texttt{tridia}, and 
\texttt{vardim}. 
And, finally, on the following problems \LBFGS{} converged to a minimizer
before generating five pairs of \LBFGS{} updates and were removed from the
test set: \texttt{eg2} and 
\texttt{tointgss}. 
This left a total of 43 problems in the test set.

The proposed method was implemented in \MATLAB{} and tested against the
\MATLAB{} \texttt{pcg} implementation of conjugate-gradients
with and without preconditioners.
The value for $\sigma$
was randomly chosen using the \MATLAB{} \texttt{rand} command.  
(Previous work has shown the \MATLAB{} ``backslash'' command is less
computationally efficient on large problems than the proposed method~\cite{ErwayMarcia12}
and are not repeated here.)
For the test using iterative methods
(i.e., \texttt{pcg}), convergence was obtained when the residual of the
linear system was less than or equal to $\sqrt{\epsilon}$, where
$\epsilon$ is machine precision in \MATLAB{} (i.e., \texttt{eps}).

\medskip

\noindent \textbf{Results.}  In the numerical experiments, \MATLAB{}
reported that \texttt{pcg} solved each linear system to the desired
accuracy without so-called ``stalling''.  In all linear solves, the
proposed recursion, \texttt{cg}, and \texttt{pcg} obtained (approximately)
the same solution.

\begin{table}
\centering
\label{table-cuter-time}
\begin{tabular}{|lc|ccc|}
  \hline
  \multicolumn{2}{|c|}
  {Problem}  &
  \multicolumn{3}{c|}
  {Time (Iterations)} \\ 
  name & $n$
  &  Recursion &\CG & \PCG (diag) \\
  \hline
 \texttt{BROYDN7D  } &  5000 &  \texttt{0.003449} & \texttt{0.006317  (6)} & \texttt{0.018552 (13) } \\

 \texttt{BRYBND    } &  5000 &  \texttt{0.003318} & \texttt{0.006485  (6)} & \texttt{0.013635 (10) } \\

 \texttt{CHAINWOO  } &  4000 &  \texttt{0.002942} & \texttt{0.006164  (6)} & \texttt{0.007573 \ (6) } \\

 \texttt{COSINE    } &  10000 &  \texttt{0.024464} & \texttt{0.010488  (6)} & \texttt{0.017255 \  (8) } \\

 \texttt{CRAGGLVY  } &  5000 &  \texttt{0.003245} & \texttt{0.066393  (4)} & \texttt{0.006829 \ (4) } \\
 \texttt{DIXMAANA  } &  3000 &  \texttt{0.002895} & \texttt{0.003395  (3)} & \texttt{0.003864 \ (3) } \\

 \texttt{DIXMAANB  } &  3000 &  \texttt{0.002555} & \texttt{0.003886  (4)} & \texttt{0.009697 \ (4) } \\

 \texttt{DIXMAANC  } &  3000 &  \texttt{0.004444} & \texttt{0.004554  (5)} & \texttt{0.005802 \ (4) } \\

 \texttt{DIXMAAND  } &  3000 &  \texttt{0.002622} & \texttt{0.004663  (5)} & \texttt{0.005995 \ (5) } \\

 \texttt{DIXMAANE  } &  3000 &  \texttt{0.002607} & \texttt{0.005161  (6)} & \texttt{0.006653 \ (6) } \\

 \texttt{DIXMAANF  } &  3000 &  \texttt{0.002610} & \texttt{0.005488  (6)} & \texttt{0.010982 \ (6) } \\

 \texttt{DIXMAANG  } &  3000 &  \texttt{0.003355} & \texttt{0.005165  (6)} & \texttt{0.007249 \ (7) } \\

 \texttt{DIXMAANH  } &  3000 &  \texttt{0.002452} & \texttt{0.005151  (6)} & \texttt{0.006591 \ (6) } \\

 \texttt{DIXMAANI  } &  3000 &  \texttt{0.002517} & \texttt{0.005291  (6)} & \texttt{0.011050 \ (6) } \\

 \texttt{DIXMAANJ  } &  3000 &  \texttt{0.002669} & \texttt{0.005991  (6)} & \texttt{0.006656  \ (6) } \\

 \texttt{DIXMAANK  } &  3000 &  \texttt{0.002594} & \texttt{0.005113  (6)} & \texttt{0.008267 \ (7) } \\

 \texttt{DIXMAANL  } &  3000 &  \texttt{0.002653} & \texttt{0.005007  (6)} & \texttt{0.006493 \ (6) } \\

 \texttt{DIXON3DQ  } &  10000 &  \texttt{0.005480} & \texttt{0.008864  (6)} & \texttt{0.013588 \ (7) } \\

 \texttt{DQDRTIC   } &  5000 &  \texttt{0.005144} & \texttt{0.006466  (5)} & \texttt{0.013477 \ (6) } \\

 \texttt{EDENSCH   } &  2000 &  \texttt{0.001783} & \texttt{0.003708  (6)} & \texttt{0.006259 \ (8) } \\

 \texttt{ENGVAL1   } &  5000 &  \texttt{0.003114} & \texttt{0.004979  (4)} & \texttt{0.007133 \ (5) } \\

 \texttt{EXTROSNB  } &  1000 &  \texttt{0.002146} & \texttt{0.002988  (6)} & \texttt{0.004553 \ (7) } \\

 \texttt{FLETCHCR  } &  1000 &  \texttt{0.001279} & \texttt{0.002611  (5)} & \texttt{0.003262 \ (5) } \\

 \texttt{FLETCBV2  } &  5000 &  \texttt{0.003049} & \texttt{0.006631  (6)} & \texttt{0.015269 (10) } \\

 \texttt{FMINSRF2  } &  5625 &  \texttt{0.003771} & \texttt{0.006565  (6)} & \texttt{0.009300 \ (7) } \\

 \texttt{FMINSURF  } &  1024 &  \texttt{0.001299} & \texttt{0.003183  (6)} & \texttt{0.004348 \ (7) } \\

 \texttt{FREUROTH  } &  5000 &  \texttt{0.005582} & \texttt{0.005960  (5)} & \texttt{0.007530 \ (5) } \\

 \texttt{GENHUMPS  } &  5000 &  \texttt{0.003390} & \texttt{0.006558  (6)} & \texttt{0.012958 (10) } \\

 \texttt{GENROSE   } &  500 &  \texttt{0.001113} & \texttt{0.002651  (6)} & \texttt{0.005155 (10) } \\
 \texttt{MOREBV    } &  5000 &  \texttt{0.003102} & \texttt{0.008279  (6)} & \texttt{0.016035 (10) } \\

 \texttt{NCB20     } &  1010 &  \texttt{0.001286} & \texttt{0.003044  (6)} & \texttt{0.005761 (10) } \\

 \texttt{NCB20B    } &  2000 &  \texttt{0.001868} & \texttt{0.003749  (6)} & \texttt{0.009726 (11) } \\

 \texttt{NONCVXU2  } &  5000 &  \texttt{0.003079} & \texttt{0.006782  (6)} & \texttt{0.009971 \ (7) } \\

 \texttt{NONCVXUN  } &  5000 &  \texttt{0.003135} & \texttt{0.006700  (6)} & \texttt{0.009887 \ (7) } \\
 \texttt{NONDQUAR  } &  5000 &  \texttt{0.003031} & \texttt{0.006796  (5)} & \texttt{0.008765 \ (6) } \\
 \texttt{POWELLSG  } &  5000 &  \texttt{0.003778} & \texttt{0.006176  (4)} & \texttt{0.027472 \ (4) } \\

 \texttt{SCHMVETT  } &  5000 &  \texttt{0.002947} & \texttt{0.045579  (6)} & \texttt{0.011899 \ (9) } \\
 \texttt{SPARSQUR  } &  10000 &  \texttt{0.005037} & \texttt{0.010826  (6)} & \texttt{0.023172 \ (9) } \\

 \texttt{SPMSRTLS  } &  4999 &  \texttt{0.003280} & \texttt{0.006679  (6)} & \texttt{0.010540 \ (8) } \\
 \texttt{SROSENBR  } &  5000 &  \texttt{0.003425} & \texttt{0.003185  (2)} & \texttt{0.003729 \ (2) } \\
 \texttt{TQUARTIC  } &  5000 &  \texttt{0.003158} & \texttt{0.003266  (2)} & \texttt{0.003805 \ (2) } \\
 \texttt{VAREIGVL  } &  1000 &  \texttt{0.001440} & \texttt{0.003064  (6)} & \texttt{0.004877 \ (8) } \\
 \texttt{WOODS     } &  4000 &  \texttt{0.008936} & \texttt{0.008285  (4)} & \texttt{0.006503 \ (4) } \\
\hline
\end{tabular}

\bigskip

\caption{Time reported to obtain the specified accuracy.
The number of iterations performed by the iterative methods is reported in parenthesis.}
\end{table}


Table 1 reports the time taken to obtain the desired accuracy.  For the
iterative methods, the total number of iterations is reported in
parenthesis after the time.  The time required by the iterative methods
often is greater than that required by the recursion formula. In fact, in
all but three experiment (\texttt{cosine}, \texttt{srosenbr}, and
\texttt{woods}), the recursion method was faster than the iterative
methods \texttt{cg} and \texttt{pcg}.

It is worth noting that L-BFGS{} matrices (i.e., $B_5$) tend to be dense
and may not be diagonally dominant. Moreover, $\sigma\in (0,1)$ may add
relatively little weight to the diagonal of $B_5$.  The results in
Table 2 suggest that a diagonal preconditioner was not
a good choice to precondition the linear system (\ref{eqn-trsp}).
In fact, using a diagonal preconditioner
often led to additional iterations with higher overall time
requirements---in part, due to the additional solve with the preconditioner.
In tests not reported here, preconditioners using the tridiagonal
of $B_5$ also failed to be a good preconditioner.


\subsection{Solving large shifted L-BFGS systems}
The second set of numerical experiments consider large
shifted \LBFGS{} systems where the shift is a tridiagonal matrix.
For these experiments we generate random symmetric tridiagonal shifts; in
particular, we solve systems of the form
\begin{equation}\label{eqn-tridiag}
(B+G)x=y,\end{equation}
as in (\ref{eqn-main}) where $G$ is a symmetric tridiagonal matrix.
Problems such as these occur in the context of interior-point methods, e.g,
see (\ref{eqn-doubly}) when $A$ is a banded upper-triangular matrix
with bandwidth two.


The elements of $G$ were chosen as follows:
\begin{equation*}
G_{i,j} = \left\{
\begin{array}{ll}
 2+\sigma+g_{ii}, \,\,\,\, & \text{if } i=j \\
 g_{ij}, & \text{if } | i-j | \leq 1,\,\, i\neq j\\
 0 ,     & \text{if } | i-j | > 1\\
\end{array}
\right.
\end{equation*}
where $\sigma$ is a positive scalar and $g_{ii}$ is randomly chosen from a
uniform distribution on $(0,1)$ and $g_{i,j}$ is randomly chosen from a
uniform distribution on $(-1,0)$.  For these experiments we chose
$\sigma=0.1$.  As in Section 5.1, the \MATLAB{} implementation of the
proposed method was tested against the \MATLAB{} \texttt{pcg}
implementation of conjugate-gradients with and without a diagonal
preconditioner.  For the test using iterative methods (i.e., \texttt{cg}
and \texttt{pcg}), convergence was obtained when the residual of the linear
system was less than or equal to $\sqrt{\epsilon}$, where $\epsilon$ is
machine precision in \MATLAB{} (i.e., \texttt{eps}).


\begin{table}
\centering
\label{table-large-resid}
\begin{tabular}{|r|ccc|}
  \hline
  {Problem}  &
  \multicolumn{3}{c|}
  {Relative Residual} \\ 
  $n$ \ \ \ \ 
  &  Recursion & \CG & \PCG (diag) \\
  \hline
  \texttt{10000}  	& \texttt{6.14e-16} & \texttt{1.13e-08} & \texttt{1.32e-08} \\ 
  \texttt{20000} 	& \texttt{6.65e-16} & \texttt{1.80e-09} & \texttt{1.25e-09} 
  \\
  \texttt{50000} 	&  \texttt{6.68e-15} & \texttt{8.25e-09} & \texttt{3.54e-09}
\\
  \texttt{100000}	&  \texttt{8.05e-16} & \texttt{4.95e-09} & \texttt{1.19e-08} 
\\
  \texttt{200000}	&  \texttt{4.71e-15} & \texttt{1.78e-10} & \texttt{7.30e-11} 
\\
  \texttt{500000}	&  \texttt{3.85e-15} & \texttt{1.41e-08} & \texttt{1.21e-08} 
\\
  \texttt{1000000} &  \texttt{3.55e-15} & \texttt{9.25e-09} & \texttt{6.61e-09} 
\\
  \texttt{2000000} &  \texttt{1.60e-14} & \texttt{1.78e-09} & \texttt{1.64e-09} 
\\
\hline
\end{tabular}

\bigskip

\caption{Relative residuals of solutions obtained using the
proposed recursion, CG, and PCG with a diagonal preconditioner.}
\end{table}

\medskip

\begin{table}
\centering
\label{table-large-time}
\begin{tabular}{|r|ccc|}
  \hline
  {Problem}  &
  \multicolumn{3}{c|}
  {Time (Iterations)} \\ 
  $n$ \ \ \ \
  &  Recursion & \CG & \PCG (diag) \\
  \hline
  \texttt{10000}  	& \texttt{0.044961} & \texttt{\ 0.030926 (14)} & \texttt{\ 0.037311 (14)} \\ 
  \texttt{20000} 	& \texttt{0.040573} & \texttt{\ 0.065083 (15)} & \texttt{\ 0.084371 (15)} \\
  \texttt{50000} 	&  \texttt{0.112462} & \texttt{\ 0.215749 (18)} & \texttt{\ 0.289687 (21)} \\
  \texttt{100000}	&  \texttt{0.290477} & \texttt{\ 0.536735 (14)} & \texttt{\ 0.522558 (13)} \\
  \texttt{200000}	&  \texttt{0.587508} & \texttt{\ 1.406732 (14)} & \texttt{\ 1.459770 (14)} \\
  \texttt{500000}	&  \texttt{1.476140} & \texttt{\ 3.313296 (13)} & \texttt{\ 4.007395 (14)} \\
  \texttt{1000000} &  \texttt{3.013683} & \texttt{\ 6.457205 (12)} & \texttt{\ 7.091486 (12)} \\
  \texttt{2000000} &  \texttt{7.437099} & \texttt{12.160872 (12)} & \texttt{13.519277 (12)} \\
\hline
\end{tabular}

\bigskip

\caption{Computational times and number of iterations obtained using the
proposed recursion, CG, and PCG with a diagonal preconditioner.}
\end{table}

\medskip
\noindent \textbf{Results.}  The recursion method, \CG{}, and \PCG{}
implemented with a diagonal preconditioner were used to solve
(\ref{eqn-tridiag}) for varying problem sizes ($1 \times 10^4 \le n \le
2\times 10^6$).  Tables 2 and 3 report results of the randomly generated
test problems; they are representative of what we have seen from many runs
of these methods.  Table 2 shows the relative residuals of the 
shifted \LBFGS{} systems with tridiagonal shifts using the three methods.
Table 3 shows the computational time for each solver.  For the iterative
methods, the total number of iterations is reported in parenthesis after the time.

All three methods were able to compute accurate solutions.  For larger
matrices, the recursion formula appears more computationally efficient than
\CG{} and \PCG.  In fact, in all runs but $n=10000$, the recursion formula
was faster than \CG{} and \PCG.  We note that these results are consistent
with our previous results for diagonal matrices $G$ (see
\cite{ErwayMarcia12, ErwayMarciaWCE12}).  In the case of $n=10000$, the
recursion formula was able to compute a more accurate solution in roughly
the same amount of time as \CG{} and \PCG.  As in Section 5.1, there are
instances where \CG{} performed better without diagonal preconditioning.
We believe that the results for \PCG{} would improve with a better
preconditioner.

\section{Conclusion}

In this paper, we proposed a direct method for solving shifted \LBFGS{}
systems that arise in unconstrained and constrained optimization.  The
recursion formula is not only able to handle very large problems ($n =
\mathcal{O}(10^6)$) for which other direct methods such as Gaussian
elimination fail, it is also very fast, being very competitive with 
conjugate gradient methods.  It is memory efficient since matrices are not
explicitly stored.  Most importantly, it is provably stable.  

The proposed recursion method can easily be applied to other types of
shifts, $G$, in \eqref{eqn-tridiag} provided systems of the form $(G +
\gamma_k^{-1}I)$ can easily solved.  For instance, if $G$ is
circulant, then solving with $(G + \gamma_k^{-1}I)$ can be efficiently
performed using fast Fourier transforms.  Such problems arise in
signal processing (see, e.g.,
\cite{waheed,MarciaWH_CSOptical,wotaoToeplitz}).  Future work
includes extensions to other quasi-Newton methods such as the
symmetric rank-1 and DFP updates as well as non-positive-definite
shifts.

\bibliographystyle{abbrv}
\bibliography{L-BFGS}

\label{lastpage}

\end{document}